\documentclass[12pt]{amsproc}
\usepackage{latexsym,amsmath,amssymb,amsfonts,amsthm,amstext,amscd}
\usepackage{tikz} 

\newcommand{\ie}{\textit{i.e.}$\,$}
\newcommand{\C}{\mathbb{C}}
\newcommand{\Quat}{\mathbb{Q}}
\newcommand{\N}{\mathbb{N}}
\newcommand{\obj}{\operatorname{obj}}
\newcommand{\id}{\operatorname{id}}

\newcommand{\End}{\operatorname{End}}

\newcommand{\cat}{\mathbf{cat}}

\theoremstyle{definition}
\newtheorem{example}{Example}[section]

\newtheorem{definition}[example]{Definition}

\theoremstyle{plain}
\newtheorem{theorem}[example]{Theorem}

\newtheorem{corollary}[example]{Corollary}

\newtheorem{lemma}[example]{Lemma}

\numberwithin{equation}{section}

\begin{document}
\title[Frobenius algebra]{$S_3$-permuted Frobenius Algebras}
\author{Zbigniew Oziewicz}
\address{Universidad Nacional Aut\'{o}noma de M\'{e}xico\\Facultdad de Estudios Superiores de Cuautitl\'{a}n\\C. P. 54714 Cuautitl\'{a}n Izcalli, Apartado Postal \#25\\
Estado de \ M\'{e}xico, \ M\'{e}xico}\email{oziewicz@unam.mx}
\author{Gregory Peter Wene}
\address{Department of Mathematics\\The University of Texas at San Antonio\\
One UTSA Circle\\San Antonio, Texas \ 78249-0624}\email{greg.wene@utsa.edu}
\thanks{Submitted August 10, 2010. Published in Proceedings of the 4th International Conference on Mathematical Sciences for Advancement of Science and Technology, Kolkata (Calcutta), India, December 2010. Institute for Mathematics, Bio-informatics, Information-technology and Computer-science, IMBIC,  www.imbic.org/index.html}
\thanks{This work is Supported by Programa de Apoyo a Proyectos de Investigaci\'on e 
Innovaci\'on Te\-cno\-l\'ogica, UNAM, Grant PAPIIT \# IN104908, 2008--2010}

\begin{abstract} In the present paper by Frobenius algebra $Y$ we mean a finite dimensional algebra $Y$ possessing an $Y$-associative and invertible (nondegenerate) form (a scalar product $\cup$), referred to as the Frobenius $Y$-structure. The nondegenerate form $\cup $ has an inverse which we will denote as $\cap $. We drop the extra conditions of associativity and unitality of $Y.$ The Frobenius algebra $\{Y,\cup,\cap\}$ determine a ternary $(3\mapsto 0)$-tensor,
$$Y\circ\cup=\cup\circ Y\quad\text{or}\quad Y _{ij}^{e}\,\cup _{ek}=\cup_{ie}\,Y^e_{jk}.$$

Frobenius algebra is formulated within the monoidal abelian category of operad of graphs $\cat(m,n).$

Frobenius algebra allows $S_2$-permuted opposite algebra to be extended to $S_3$-permuted algebras. If $\sigma\in S_3$ denotes a permutation we can get a $\sigma$-permuted algebra given by 
$$\{Y,\cup,\cap\}\quad\xrightarrow{\quad\sigma\quad}\quad\cap\circ\sigma(Y\circ\cup).$$

Operad of graphs, \ie diagrammatic language, is used both to illustrate the construction as well
as a method of proof for the main Theorem. We give two detailed examples of
this construction for Clifford algebras. Our construction, however, applies to all Frobenius algebras.\end{abstract}\maketitle
\newpage\tableofcontents

\section{Semifields and projective planes, by Knuth in 1965} Knuth [Knuth 1965] realized that the multiplication constants $Y_{ij}^{k}$ of an n-dimensional k-algebra $Y\in\cat(2,1)$ determine a $n\times n\times n$ cube $A$, the cube associated with the $k$-algebra $Y,$ and examines the algebras that arise when the axes of the cube are permuted. Let $\sigma\in S_3$ is a permutation. Then $A^\sigma$ will represent the 3-cube $A$ with subscripts permuted by $\sigma.$ 

Knuth [Knuth 1965] showed that if $\pi $ is a finite projective plane coordinatized by the division ring $S$ and $A$ is the cube associated with $S,$ then the six permutations of the indices of $A$ determine a series of at most six planes.

The algebra corresponding to $A^{\left( 12\right) }$ is called the opposite
algebra and is denoted by $Y^{op}$. Pairing each algebra with its opposite we have
$$\begin{tabular}{c|c}$Y\simeq A$ & $Y^{\text{op}}\simeq A^{(23)}$\\\hline\\[-7pt]$A^{(123)}$ & $A^{(23)}=A^{(123)(12)}$ \\ \hline\\[-7pt]
$A^{\left( 132)\right) }$ & $A^{(13)}=A^{(132)(12)}$ \\\hline
\end{tabular}$$\smallskip

We note that every algebra $Y\in\cat(2,1)$ is a mixed tensor, 2-covariant and 1-contra-variant, and every $S_3$-permutation of $Y$ involve implicitly the existence of the non-degenerate scalar product. Moreover in order to assure that $S_3$-permutation leads to the unique permuted new algebra $\sigma Y\,\in\cat(2,1),$ this implicit scalar product must be necessarily $Y$-associative. This is equivalent to say that an algebra $Y\in\cat(2,1)$ must be a Frobenius algebra. In what follows we will assume that an algebra $Y\in\cat(2,1)$ is a Frobenius algebra.  

Frobenius algebra is usually defined to be an associative and unital algebra $Y$ possessing a (left or right) $Y$-module isomorphism, or, equivalently, possessing a $Y$-associative invertible scalar product, denoted here by  $\cup\in\cat(2,0),$ called the Frobenius structure or the Frobenius pairing, see e.g. [Eilenberg and Nakayama 1955; Caenepeel, Militaru, Zhu 2002, Definition 3 on page 32; Kock 2003, pages 95--97, Definition 2.2.5].
 
Let $Y\in\cat(2,1)$ denote a finite dimension k-algebra, not necessarily associative, equipped with an $Y$-associative nondegenerate form $\cup\in\cat(2,0).$ The ternary tensor (the ternary scalar product), $Y\circ\,\cup=\cup\circ\,Y\in\cat(3,0),$ is defined in terms of the multiplication tensor $Y _{ij}^{k}$ and the  form $\cup\in\cat(2,0),$
\begin{equation}
(Y\circ\cup) _{ijk}=Y _{ij}^{e}\cup _{ek}.\end{equation}

Since the  form $\cup$ is nondegenerate, it has an inverse, $\cap\in\cat(0,2),$ and 
\begin{gather}\cup\circ\cap=\cap\circ\cup=\arrowvert\quad\in\cat(1,1),\end{gather}
\begin{eqnarray}Y _{ij}^{f}\cup _{fk}\cap ^{ke} &=&(Y\circ\cup) _{ijk}\cap ^{ke} \\
Y _{ij}^{e} &=&(Y\circ\cup) _{ijk}\cap ^{ke}.\end{eqnarray}

We can permute the ternary scalar tensor to get
\begin{gather}(\sigma Y)_{ij}^{k}=(Y\circ\cup) _{\sigma(ije)}\cap^{ek}\end{gather}
The multiplication tensor $(\sigma Y)_{ij}^{k}$ will determine a new $k$-algebra, that needs not to be necessarily a Frobenius algebra. 

\section{Some Nonassociative Algebra Preliminaries} We work within monoidal abelian category generated by a single object. 

Since the algebras we will be working with are not necessarily associative,
we will use the associator of three elements $a,b,c$ of the algebra $Y,$
\begin{gather}\left( a,b,c\right) =(ab)c-a(bc)\quad\in\cat(3,1).\end{gather}
The above associator needs that a monoidal category must be abelian.

The left nucleus of $Y$ is the set 
\begin{gather}N_{l}=\left\{ l\in Y:\left( l,b,c\right) =0\text{ for all }b,c\in \right\}.\end{gather}

The middle nucleus of $Y$ is the set 
\begin{gather}N_{m}=\left\{ m\in Y:\left(a,m,c\right) =0\text{ for all }a,c\in Y\right\}.\end{gather}
The right nucleus of $Y$ is the set \begin{gather}N_{r}=\left\{ r\in Y:\left(
a,b,r\right) =0\text{ for all }a,b\in Y\right\}.\end{gather} The nucleus is
the set $N=N_{l}\cap N_{m}\cap N_{r}$. 

We will denote the associator in the $\sigma$-permuted Frobenius algebra $\sigma Y,$ $\sigma \neq\id$, by $\left( a,b,c\right)^{\circ }.$

An algebra $Y\in\cat(2,1)$ is said to be flexible if for all $a,b\in Y,$\begin{equation}
\left( a,b,a\right) =0.\end{equation}
The flexible property seems to be a minimum
requirement in the existing studies of nonassociative real division algebras
[Althoen and Kugler 1983, Benkart, Britten and Osborn 1982, Darpo 2006] and Malcev-admissible algebra [Myung 1986]. The associative, alternative and Jordan algebras
all enjoy the flexible property. An even weaker association property is the
association of cubes for all $a\in Y,$
\begin{gather}\left( a,a,a\right) =0.\end{gather}

Clearly, the flexible property implies the association of cubes. The association of cubes appears to be a "minimal" regularity condition that permits the analysis of an algebra. In
his article Osborn [Osborn 1972] begins the study of identities on
noncommutative algebras insisting that the algebras satisfies the
association of cubes. We will see that the nonassociative Frobenius algebras
do not satisfy this very weak associativity condition.

The concept of isotopism was introduced by Albert to provide a broader
classification of the many nonassociative algebras than that of isomorphism
[Tomber 1979]. Algebras $Y$ and $Y%
^{\prime }$ with products $\ast $ and $\circ $ are isotopic if and only if
there exists nonsingular linear transformations $F,G,H$ from $Y$
to $Y^{\prime }$ such that
\begin{gather} (a\circ b)H=aF\ast bG.\end{gather}

\noindent The relation "isotopic to" is an equivalence relation.

Clearly, an isotope of a division algebra is a division algebra. An algebra
is said to be isotopically simple if every isotope of it is simple. We note
that any isotope of a simple Clifford algebra will again be simple.

It is well known [Knuth 1965] that we can make a division
algebra without a unit element into a division algebra with unit element $e$
under a new product

\begin{gather}\left( a\circ e\right) \star \left( e\circ b\right) =a\circ b\text{.}\end{gather}

The classical introduction to the theory of nonassociative algebras is the
book by Schafer [Schafer 1966]. The article [Tomber 1979]
gives a history of nonassociative algebras.

\section{Main Results} We work within a bi-closed monoidal category (another name a tensor category) generated by a single object. In this case the set of all objects coincides with the set of non-negative integers (with the set of natural numbers) $\N.$ Thus the set of all objects is, $\obj\cat=\N,$ and $1\in\N$ is a generating object. This leads to graphical language where
each morphism (an arrow of a category) is characterized by a \textit{pair} of non-negative integers, morphims are bi-graded, and we refer to this pair  \{input, output\} = \{entrance, exit\}, as to the type or arity of the morphism = \{arity-in, arity-out\},
\begin{gather}\N\ni m\quad\xrightarrow{\quad\text{morphisms of $(m\rightarrow n)$-arity\quad}}\quad n\in\N.\end{gather}

\begin{definition}[Frobenius algebra] A not necessarily associative and not necessarily unital algebra $Y\in\cat(2,1)$ is said to be a Frobenius algebra if exists a morphism $\cup\in\cat(2,0)$ such that the following two conditions hold
\begin{gather}Y\circ\cup=\cup\circ Y\quad\in\cat(3,0)\label{F0}\\
\cup\circ\cap=\cap\circ\cup=\arrowvert=\id\quad\in\cat(1,1).\end{gather} 
Within graphical or graphics language, the first of the above Frobenius condition (a solvable Frobenius algebra if $\cup\neq 0$) is precisely the relation among two morphisms from $\cat(3,0)$ as follows\bigskip
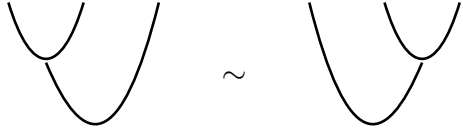
\begin{center}
\begin{tikzpicture}[line width=1pt]
\draw(-2,0) .. controls (-1.7,-1) and (-1.3,-1) .. (-1,0);
\draw(-1.5,-0.8)..controls(-1,-2) and (-0.5,-2)..(0,0);
\draw(1,-1) node {$\sim$};
\draw(3,0) .. controls (3.3,-1) and (3.7,-1) .. (4,0);
\draw(3.5,-0.8)..controls(3,-2) and (2.5,-2)..(2,0);
\end{tikzpicture}\end{center}
\end{definition}

\begin{theorem}\label{main} A necessary condition that the $k$-algebra $Y$ has nondegenerate $Y$-associative form $\cup$ is the trace of the matrix for right multiplication by every element $x$ of $Y$ is the same as the trace of the matrix for left multiplication by the element $x.$ This must hold for the traces of every power of the regular, left- and right-, representations. 
\end{theorem}
\begin{proof} One can compose \eqref{F0} on both sides with inverse of the Frobenius structure, $\cap\equiv\cup^{-1},$
\begin{gather}Y\circ\cup\circ\cap=Y\circ\arrowvert=Y=\cup\circ Y\circ\cap,\\
\cap\circ\cup\circ Y=\arrowvert\circ Y=Y=\cap\circ Y\circ\cup,\\
\cap\circ Y\circ\cup=Y=\cup\circ Y\circ\cap.\label{F1}\end{gather}
The last conditions \eqref{F1} within operad of graphs is precisely the following graphical relations
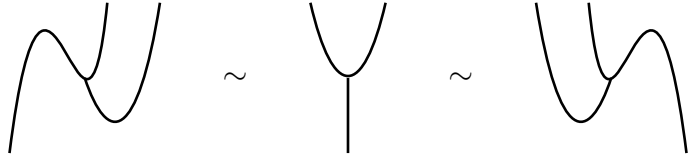
\vspace{-20mm}
\begin{gather}\begin{tikzpicture}[line width=1pt]
\draw(-0.5,1)..controls(-0.2,-0.3) and (0.2,-0.3)..(0.5,1);
\draw(0,0)--(0,-1);
\draw(-1.5,0) node {$\sim$};\draw(1.5,0) node {$\sim$};
\draw(-4.5,-1)..controls(-4,3) and (-3.5,-2)..(-3.2,1);
\draw(4.5,-1)..controls(4,3) and (3.5,-2)..(3.2,1);
\draw(-3.5,0)..controls(-3.2,-0.9) and (-2.8,-0.9)..(-2.5,1);
\draw(3.5,0)..controls(3.2,-0.9) and (2.8,-0.9)..(2.5,1);
\end{tikzpicture}\label{F2}\end{gather}

Every algebra $Y$ can be seen as a right $Y$-module and as a left $Y$-module, in \textit{two} ways, on both dual objects. An associative algebra $Y$ is a two-sided $Y$-module, known as $Y$-bimodule, and this $Y$-bimodule structure is again doubled on both dual objects. 

An element of an algebra $Y$ can be seen as a morphism $\in\cat(0,1).$ The left- and the right- composition of $Y\in\cat(2,1)$ with $\cat(0,1)$ gives the two regular representations,
\begin{gather}\cat(0,1)\circ\cat(2,1)\quad\subset\quad\cat(1,1)\simeq\End(1,1),\\
\cat(2,1)\circ\cat(0,1)\quad\subset\quad\cat(1,1)\simeq\End(1,1).\end{gather}
\begin{center}\begin{tikzpicture}[line width=1pt]
\draw(-3.3,0)--(-3.3,-1);\draw(-3.3,0)node{$\bullet$};
\draw(-2.5,0)node{$\circ$};\draw(-2,1)..controls(-1.7,-0.1)and(-1.3,-0.1)..(-1,1);
\draw(-1.5,0.2)--(-1.5,-1);\draw(0,0)node{$\equiv$};
\draw(1,0.5)..controls(1.3,-0.2)and(1.7,-0.2)..(2,1);
\draw(1.5,0)--(1.5,-1);\draw(1,0.5)node{$\bullet$};
\end{tikzpicture}\end{center}

\begin{center}\begin{tikzpicture}[line width=1pt]
\draw(-1,0)--(-1,-1);\draw(-1,0)node{$\bullet$};
\draw(-2,0)node{$\circ$};\draw(-4,1)..controls(-3.7,-0.1)and(-3.3,-0.1)..(-3,1);
\draw(-3.5,0.2)--(-3.5,-1);\draw(0,0)node{$\equiv$};
\draw(1,1)..controls(1.3,-0.2)and(1.7,-0.2)..(2,0.5);
\draw(1.5,0)--(1.5,-1);\draw(2,0.5)node{$\bullet$};
\end{tikzpicture}\end{center}

Composing $\cat(0,1)$ with \eqref{F1}-\eqref{F2} we arrive to the following relations among endomorphisms within $\cat(1,1),$
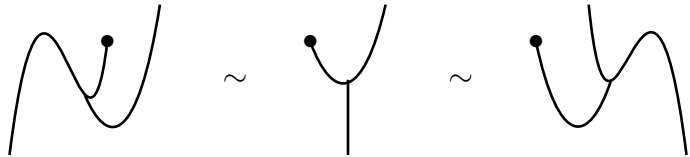
\vspace{-15mm}
\begin{center}\begin{tikzpicture}[line width=1pt]
\draw(-0.5,0.5)..controls(-0.2,-0.3) and (0.2,-0.3)..(0.5,1);
\draw(-0.5,0.5)node{$\bullet$};
\draw(0,0)--(0,-1);
\draw(-1.5,0) node {$\sim$};\draw(1.5,0) node {$\sim$};
\draw(-4.5,-1)..controls(-4,3) and (-3.5,-2)..(-3.2,0.5);
\draw(-3.2,0.5)node{$\bullet$};
\draw(4.5,-1)..controls(4,3) and (3.5,-2)..(3.2,1);
\draw(-3.5,-0.2)..controls(-3.2,-0.9) and (-2.8,-0.9)..(-2.5,1);
\draw(3.5,0)..controls(3.2,-0.9) and (2.8,-0.9)..(2.5,0.5);
\draw(2.5,0.5)node{$\bullet$};
\end{tikzpicture}\label{F3}\end{center}
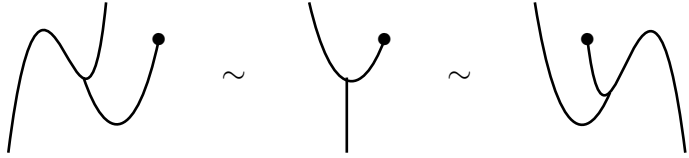
\vspace{-25mm}
\begin{gather}\begin{tikzpicture}[line width=1pt]
\draw(-0.5,1)..controls(-0.2,-0.3) and (0.2,-0.3)..(0.5,0.5);
\draw(0.5,0.5)node{$\bullet$};
\draw(0,0)--(0,-1);
\draw(-1.5,0) node {$\sim$};\draw(1.5,0) node {$\sim$};
\draw(-4.5,-1)..controls(-4,3) and (-3.5,-2)..(-3.2,1);
\draw(4.5,-1)..controls(4,3) and (3.5,-2)..(3.2,0.5);
\draw(3.2,0.5)node{$\bullet$};
\draw(-3.5,0)..controls(-3.2,-0.9) and (-2.8,-0.9)..(-2.5,0.5);
\draw(-2.5,0.5)node{$\bullet$};
\draw(3.5,-0.2)..controls(3.2,-0.9) and (2.8,-0.9)..(2.5,1);
\end{tikzpicture}\label{F4}\end{gather}

Our next aim is to calculate the traces in a ring $k\equiv\cat(0,0),$
\begin{gather}\text{trace}(\ldots)\equiv\text{evaluation}\circ(\ldots)\circ\text{co-evaluation}.\end{gather}
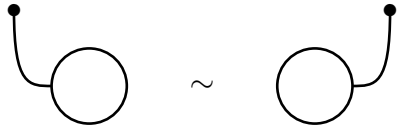
\begin{gather}\begin{tikzpicture}[line width=1pt]
\draw(-5.5,0)circle(0.5cm);
\draw(-2.5,0)circle(0.5cm);
\draw(-4,0)node{$\sim$};
\draw(-6.5,1)node{$\bullet$};
\draw(-6.5,1)..controls(-6.5,0)and(-6.3,0)..(-6,0);
\draw(-1.5,1)node{$\bullet$};
\draw(-1.5,1)..controls(-1.5,0)and(-1.7,0)..(-2,0);
\end{tikzpicture}\label{F5}\end{gather}
Two circles in \eqref{F5} are $\cup$-free and they are given in terms of evaluation and co-evaluation  of the dual objects, \ie in terms of the bi-closed structure [Kelly and Laplaza 1980]. 

The relation \eqref{F5} is desired equality of traces of the regular representations, and this is $\cup$-free, \ie it is the necessary condition for an algebra $Y$ to be the Frobenius algebra. 

We left to the reader the graphical proof that not only traces of the regular representations of the Frobenius algebra must be equal, but as well as the traces of all \textit{powers} of regular representations must be the same. This proves the theorem about the necessary condition for an algebra $Y\in\cat(2,1)$ to be the Frobenius algebra.\end{proof}

\begin{corollary} The relations \eqref{F4} leads also to the following consequence, that is useless for our searching of the necessary condition for an algebra $Y\in\cat(2,1)$ to be a Frobenius algebra, in terms of $Y$ alone. The following relation within $\cat(0,0)$ involve the Frobenius structure besides of an algebra 
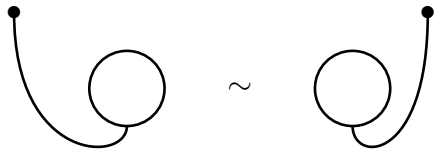
\begin{gather}\begin{tikzpicture}[line width=1pt]
\draw(2.5,0)circle(0.5cm);
\draw(5.5,0)circle(0.5cm);
\draw(4,0)node{$\sim$};
\draw(1,1)node{$\bullet$};
\draw(1,1)..controls(1,-1)and(2.5,-1)..(2.5,-0.5);
\draw(6.5,1)node{$\bullet$};
\draw(5.5,-0.5)..controls(5.5,-1)and(6.5,-1)..(6.5,1);
\end{tikzpicture}\label{F6}\end{gather}
Note that the circles in \eqref{F5} do not have the same meaning as in \eqref{F6}. The two circles in \eqref{F6} are $(\cap,\cup)$-dependent, \ie the animals in \eqref{F6} involve the both Frobenius structures, $\cup$ and $\cap.$\end{corollary}

\begin{corollary} Another extra bonus corollary that follows from \eqref{F4} is the following relation
\begin{gather}\begin{tikzpicture}[line width=1pt]
\draw(-2.5,0)circle(0.5cm);
\draw(-1.5,1)node{$\bullet$};
\draw(-1.5,1)..controls(-1.5,0)and(-1.7,0)..(-2,0);
\draw(0,0)node{$\sim$};
\draw(2.5,0)circle(0.5cm);
\draw(1,1)node{$\bullet$};
\draw(1,1)..controls(1,-1)and(2.5,-1)..(2.5,-0.5);
\end{tikzpicture}\label{F7}\end{gather}
\end{corollary}

\begin{example}[Kock 2003] The algebra of real upper triangular matrices does not
admit a nondegenerate associative  form [Kock 2003].
Letting $e_{11}$, $e_{11}$, and $e_{11}$ be an ordered basis for this
algebra, we see that the matrix for right multiplication by $e_{11}$ is
\begin{gather}
\left[ \begin{array}{ccc}
1 & 0 & 0 \\ 
0 & 0 & 0 \\ 
0 & 0 & 0%
\end{array}%
\right] \end{gather}%
and the matrix \ for left multiplication by $e_{11}$ is

\begin{gather}
\left[ 
\begin{array}{ccc}
1 & 0 & 0 \\ 
0 & 1 & 0 \\ 
0 & 0 & 0%
\end{array}%
\right] .
\end{gather}%
The traces are different.\end{example}

\section{Other results}
\begin{lemma}\label{l1} Let $Y$ denote a finite dimensional Frobenius $k-$algebra with
the Frobenius structure $\cup$ that is diagonal matrix in some basis. The algebra $Y^{(12)}$ is the algebra $Y^{op}$.\end{lemma}
\begin{proof}
\begin{gather}\begin{aligned}
(Y\circ\cup) _{ijk} &=Y _{ij}^{k}\cup _{kk}. \\
(\sigma Y)_{ij}^{k} &=(Y\circ\cup) _{jik}\cap ^{kk}=Y _{ji}^{k}\cup
_{kk}\cap ^{kk}. \\
(\sigma Y)_{ij}^{k} &=Y _{ji}^{k}
\end{aligned}\end{gather}
\end{proof}

\begin{lemma}
The algebra $Y^{\left( 23\right) }$will always have a left
identity and will have an identity if and only if $\cup $ is the identity
matrix.
\end{lemma}

\begin{proof}
The computation for the left identity goes as follows:

\begin{gather}
(Y\circ\cup) _{i1i}=Y _{i1}^{i}\cup _{ii}=\cup _{ii}.
\end{gather}

\begin{gather}
(\sigma Y)_{1i}^{i}=(Y\circ\cup) _{i1i}\cap ^{ii}=\cup _{ii}\cap ^{ii}=1
\end{gather}

since $\cup $ and $\cap $ are inverses and $\cup $ is a diagonal matrix.

The algebra $Y^{\left( 23\right) }$will have a right identity if
and only if%
\begin{gather}
(\sigma Y)_{i1}^{i}=(Y\circ\cup) _{ii1}\cap ^{11}=1\text{.}
\end{gather}%
But $(Y\circ\cup) _{ii1}=Y _{ii}^{1}\cup _{11}=\cup _{ii}$ and then

\begin{gather}
(\sigma Y)_{i1}^{i}=\cup _{ii}\cap ^{11}=1\text{.}
\end{gather}%
This last forces $\cup _{ii}=1$ for all $i$.
\end{proof}

\begin{theorem}
Cubes do not associate in the algebra $Y^{\left( 23\right) }$ if
there is some $\cup _{xx}=-1$.
\end{theorem}

\begin{proof}
We compute the associator $\left( b_{x},b_{x},b_{x}\right) ^{\circ }$.

\begin{gather}
b_{x}\circ b_{x}=(\sigma Y)_{xx}^{k}=(Y\circ\cup) _{xkx}\cap ^{xx}\text{.}
\end{gather}%
Since $(Y\circ\cup) _{xkx}=Y _{xk}^{x}\cup _{xx}$ and $Y _{xk}^{x}$ is nonzero
if and only if $k=1$ in which case it has the value 1, $b_{x}\circ b_{x}=-E$.

\begin{gather}\begin{aligned}
\left( b_{x},b_{x},b_{x}\right) ^{\circ } &=\left( b_{x}\circ b_{x}\right)
\circ b_{x}-b_{x}\circ \left( b_{x}\circ b_{x}\right) \\
&=-E\circ b_{x}+b_{x}\circ E
\end{aligned}\end{gather}
By the above Lemma we know that $-E\circ b_{x}=-b_{x}$

$(\sigma Y)_{x1}^{k}=(Y\circ\cup) _{xk1}\cap ^{11}$ where $(Y\circ\cup) _{xk1}=Y
_{xk}^{1}\cup _{11}$as before $Y _{xk}^{1}$ is nonzero if and only if $%
k=x $ in which case it has the value $-1$, and $b_{x}\circ E=-bx$.
\end{proof}

The algebra $Y^{\left( 13\right) }$ will always have a right
identity element; we omit the corresponding computations for this algebra.

\section{The Complex Numbers} The complex numbers are the Clifford algebra $Cl_{0,1}$ with
multiplication table given by
\begin{gather}\begin{tabular}{c|c|c}
& $E$ & $I$ \\ \hline
$E$ & $E$ & $I$ \\ \hline
$I$ & $I$ & $-E$\end{tabular}\\
Y _{11}^{1}=Y _{12}^{2} =Y _{21}^{2} =1,\quad Y _{22}^{1} =-1\end{gather}

This algebra possesses a family of Frobenius structures generated by two forms 
\begin{gather}\begin{pmatrix}1&0\\0&-1\end{pmatrix}\quad\text{and}\quad
\begin{pmatrix}0&1\\1&0\end{pmatrix}\end{gather}

The one nondegenerate $\C$-associative  form $\cup$ is $\begin{pmatrix}
1 & 0 \\ 0 & -1\end{pmatrix}$ and 
\begin{gather}\cap =\cup ^{-1}\simeq\cup.\end{gather} We want to compute the algebra 
$Cl_{0,1}^{\left( 23\right) }$

The ternary tensor is $(Y\circ\cup) _{ijk}=Y _{ij}^{e}\cup _{ek}=Y
_{ij}^{k}\cup _{kk}$.

$(Y\circ\cup) _{111}=Y _{11}^{1}\cup _{11}=1$

$(Y\circ\cup) _{122}=Y _{12}^{2}\cup _{22}=-1$

$(Y\circ\cup) _{212}=Y _{21}^{2}\cup _{22}=-1$

$(Y\circ\cup) _{221}=Y _{22}^{1}\cup _{11}=-1$.\smallskip

The new multiplication constants will be $(\sigma Y)_{ij}^{k}=(Y\circ\cup)
_{ikj}\cap ^{jj}$.

$(\sigma Y)_{11}^{1}=(Y\circ\cup) _{111}\cap ^{11}=1\times 1=1$

$(\sigma Y)_{12}^{2}=(Y\circ\cup) _{122}\cap ^{22}=-1\times -1=1$

$(\sigma Y)_{21}^{2}=(Y\circ\cup) _{221}\cap ^{11}=-1\times 1=-1$

$(\sigma Y)_{22}^{1}=(Y\circ\cup) _{212}\cap ^{22}=-1\times -1=1$
\medskip

The multiplication in the new algebra will be denoted by $\circ.$
The multiplication table will be given by

\begin{gather}\begin{tabular}{c|c|c}
$\circ $ & $E$ & $I$ \\ \hline
$E$ & $E$ & $I$ \\ \hline
$I$ & $-I$ & $E$\end{tabular}\end{gather}

This algebra has a left identity, $E$; furthermore, 
\begin{gather}\left( \alpha E+\cup
I\right) \circ \left( \alpha E+\cup I\right) =\left( \alpha ^{2}+\cup
^{2}\right) E.\end{gather}

The left nucleus is the field generated by the element $E$; 
\begin{gather}N_{m}=N_{r}=(0)\end{gather}

This algebra does not support a nondegenerate associative  form $B$
because the matrix for right multiplication by $E$ is $\left[ 
\begin{array}{cc}
1 & 0 \\ 
0 & -1
\end{array}
\right] $ and the matrix for left multiplication by $E$ is $\left[ 
\begin{array}{cc}1 & 0 \\ 0 & 1
\end{array}\right] $; the traces are not the same.

We make this algebra into the complex numbers under the product
\begin{gather}\left(a\circ e\right) \star \left( e\circ b\right) =a\circ b.\end{gather}
Any division algebra 2-dimensional over the field $\mathbb{R}$ will be the complex numbers [Dickson 1906].

\section{The Quaternion Division Algebra} The multiplication tensor of the real quaternion division algebra $Cl_{0,2}$ is given in
\begin{gather}
\begin{tabular}{c|c|c|c|c}& E & I & J & K \\ \hline
E & E & I & J & K \\ \hline
I & I & -E & K & -J \\ \hline
J & J & -K & -E & I \\ \hline
K & K & J & -I & -E\end{tabular}\end{gather}

The nondegenerate associative  form $B$ is given by 
\begin{gather}\cup
_{ij}=\left\{ \begin{array}{l}b_{i}^{2} \\ 0\end{array}\right. 
\begin{array}{l}if\text{ }i=j \\ \text{otherwise}\end{array}\end{gather}

The five new algebras that arise form our construction are given below. Each
can be shown to be an isotope of the quaternion algebra or its opposite
under the product

\begin{gather}
\left( a\circ e\right) \star \left( e\circ b\right) =a\circ b\text{.}
\end{gather}%
The relations among the various algebras follow those given in the table in
the Introduction. Only the quaternion algebra and its opposite are
associate. The four nonassociative algebras satisfy

\begin{gather}
\left( \alpha _{0}E+\alpha _{1}I+\alpha _{2}J+\alpha _{3}K\right)
^{2}=\left( \alpha _{0}^{2}+\alpha _{1}^{2}+\alpha _{2}^{2}+\alpha
_{3}^{2}\right) E
\end{gather}%
for all $\left( \alpha _{0}E+\alpha _{1}I+\alpha _{2}J+\alpha _{3}K\right) $
in the algebra. The nonassociative algebras posses only a one sided
identity. There is no in-depth study of the nonflexible division algebras in
the literature.

\begin{gather}
\underset{\sigma =(12)}{%
\begin{tabular}{c|c|c|c|c}$\circ $ & E & I & J & K \\ \hline
E & E & I & J & K \\ \hline
I & I & -E & -K & J \\ \hline
J & J & K & -E & -I \\ \hline
K & K & -J & I & -E \\ \hline
\end{tabular}%
}\underset{\sigma =(23)}{\text{ \ \ \ \ \ \ \ \ \ \ }%
\begin{tabular}{c|c|c|c|c}$\circ $ & E & I & J & K \\ \hline
E & E & I & J & K \\ \hline
I & -I & E & K & -J \\ \hline
J & -J & -K & E & I \\ \hline
K & -K & J & -I & E \\ \hline
\end{tabular}%
}
\end{gather}%
\begin{gather}
\underset{\sigma =(123)}{%
\begin{tabular}{c|c|c|c|c}$\circ $ & E & I & J & K \\ \hline
E & E & I & J & K \\ \hline
I & -I & E & K & -J \\ \hline
J & -J & -K & E & I \\ \hline
K & -K & J & -I & E \\ \hline
\end{tabular}%
}\ \ \ \ \ \ \ \ \ \ \ \ \underset{\sigma =(13)}{%
\begin{tabular}{c|c|c|c|c}$\circ $ & E & I & J & K \\ \hline
E & E & -I & -J & -K \\ \hline
I & I & E & -K & J \\ \hline
J & J & K & E & -I \\ \hline
K & K & -J & I & E \\ \hline
\end{tabular}%
}
\end{gather}%

\begin{gather}
\underset{\sigma =(132)}{\begin{tabular}{c|c|c|c|c}$\circ $ & E & I & J & K \\ \hline
E & E & -I & -J & -K \\ \hline
I & I & E & K & -J \\ \hline
J & J & -K & E & I \\ \hline
K & K & J & -I & E \\ \hline
\end{tabular}}\end{gather}

Lemma \ref{l1} applied to the element $E$ in each of the nonassociative
algebras impies that none has admits an $\Quat$-associative, nondegenerate 
form.

\section{Conclusion} Theorem \ref{main} gives a necessary condition for the existence of a nondegenerate $Y$-associative form $\in\cat(2,0)$ in terms of the $Y\in\cat(2,1)$ alone. Can we find a sufficient condition using only the algebra $Y$?

The examples beg the question: "Are all 4-dimensional real division algebras
in which each element squares to a scalar multiple of the one sided identity
an isotope of the quaternion division algebra?"

The concepts developed in the paper can be applied to any finite dimensional
Frobenius algebra $\{Y,\cup,\cap\}.$

\end{document}